\documentclass[10pt, reqno]{amsart}
\usepackage{amsmath} 
\usepackage{amsthm}
\usepackage{amssymb}
\usepackage{amsfonts}
\usepackage{mathdots}
\usepackage{latexsym}
\usepackage{caption}
\usepackage{subcaption}
\usepackage{graphicx} 
\usepackage{setspace} 
\usepackage{hyperref}
\usepackage{cite} 
\usepackage{multicol} 
\usepackage[table]{xcolor}

\usepackage[sc]{mathpazo}
\usepackage[T1]{fontenc}

\newcommand{\C}{\mathbb{C}}

\newcommand{\Z}{\mathbb{Z}}

\newcommand{\floor}[1]{\lfloor #1 \rfloor}

\newcommand{\minimatrix}[4]{\begin{bmatrix} #1 & #2 \\ #3 & #4 \end{bmatrix}  }

\theoremstyle{plain}

\newtheorem{Lemma}{Lemma}
\theoremstyle{definition}

\allowdisplaybreaks

\begin{document}

    \title[On the matrix equation $XA+AX^T =0$, II]{On the matrix equation $XA+AX^T =0$, II:\\Type 0-I interactions}

    \author[Alice Chan]{Alice Zhuo-Yu Chan}

    \author[L.A.~Garcia German]{Luis Alberto Garcia German}

    \author{Stephan Ramon Garcia}
    \email{Stephan.Garcia@pomona.edu}
    \urladdr{\url{http://pages.pomona.edu/~sg064747}}

    \author{Amy L. Shoemaker}

    \address{   Department of Mathematics\\
            Pomona College\\
            Claremont, California\\
            91711 \\ USA}


    \thanks{Partially supported by National Science Foundation Grants DMS-1001614 and DMS-1265973.}

    \begin{abstract}
	The matrix equation $XA + AX^T = 0$ was recently introduced by
	De Ter\'an and Dopico (Linear Algebra Appl. {\bf 434}~(2011), 44--67) to study the dimension of congruence
	orbits.  They reduced the study of this equation to a number of special cases, several of which have not been
	explicitly solved.  In this note we obtain an explicit, closed-form solution in the difficult Type 0-I interaction case.
    \end{abstract}

\maketitle
 
%
%
%
%
\section{Introduction}

	The matrix equation
	\begin{equation}\label{eq-DTD}
		XA+AX^T = 0,
	\end{equation}
	where $A$ in $M_n(\C)$ is fixed and $X$ is unknown, was introduced in 2011 by F.~De Ter\'an and F.M.~Dopico
	to study the dimensions of congruence orbits.  In this setting, the codimension of the congruence orbit of $A$ is given by 
	the number of free parameters in the solution of \eqref{eq-DTD}.  More recently, the equation \eqref{eq-DTD}
	has also attracted the attention of a number of other authors \cite{Dmytryshyn, Dmytryshyn2, Dmytryshyn3, OMEXAAX}.		
	
	Now observe that for each fixed $A$, the solution space to \eqref{eq-DTD} is
	a Lie algebra, denoted $\mathfrak{g}(A)$, equipped with Lie bracket $[X,Y] =XY - YX$.   Building upon the initial 
	work of De Ter\'an and Dopico, a complete description of such twisted matrix Lie algebras is almost at hand.  
	The computations in \cite{DTD} provide explicit descriptions, up to similarity, of all possible twisted matrix Lie algebras
	except in a few highly problematic cases.  In \cite{OMEXAAX}, the third and fourth authors provided an explicit description of the solution set in
	one of the most difficult cases.  

	We are concerned here with the particularly troublesome setting where $A$ is the direct sum of a so-called 
	Type 0 and Type I canonical matrix (see \cite[Thm.~4.5.25]{HJ}, \cite[Sect.~2]{DTD}).
	These are matrices of the form $J_p(0)\oplus \Gamma_q$, where $J_p(0)$ is the $p \times p$ 
	Jordan matrix with eigenvalue $0$ and
	$\Gamma_q$ is the $q\times q$ matrix
	\begin{equation*}
		\Gamma_q :=\small
		\begin{bmatrix}
			0&0&0&\cdots &0& (-1)^{q+1}\\
			0&0&0&\cdots &(-1)^q &(-1)^q\\
			\vdots&\vdots & \vdots & \ddots &\vdots &\vdots\\
			0&0&1&\cdots &0&0\\
			0&-1&-1&\cdots &0&0\\
			1&1&0&\cdots &0 &0
		\end{bmatrix}.
	\end{equation*}
	
	We aim here to provide an explicit, closed form solution to the matrix equation
	\begin{equation}\label{eq-Main}	
		X\big(J_p(0)\oplus \Gamma_q\big)  +  \big(J_p(0)\oplus \Gamma_q\big) X^T=0,
	\end{equation}
	thereby obtaining a simple matricial description of the Lie algebra $\mathfrak{g}\big(J_p(0)\oplus \Gamma_q\big)$.
	For even $p$, it is not difficult to show that
	\begin{equation*}
		\mathfrak{g}(J_p(0)\oplus \Gamma_q)=\mathfrak{g}\big(J_p(0)\big)\oplus \mathfrak{g}(\Gamma_q),
	\end{equation*}
	where the Lie algebras $\mathfrak{g}(J_p(0))$ and $\mathfrak{g}(\Gamma_q)$ are elegantly described
	in \cite[p.~52, p.~54]{DTD}.   For odd $p$, however,
	the situation is much more complicated.  Indeed, symbolic computation reveals
	that the matrices which arise are bewilderingly complex and typographically unwieldy.
	The consideration of this troublesome case takes up the bulk of this article.
	
	Besides providing a major step in the description of twisted matrix Lie algebras, our approach
	is notable for the use of a number of combinatorial identities.  In fact, we conclude the paper by outlining
	an alternate description of the solutions to \eqref{eq-Main} in which a variety of intriguing combinatorial quantities emerge.

\section{Solution}
	
	We begin our consideration of the matrix equation \eqref{eq-Main} by partitioning the unknown $(p+q)\times (p+q)$ matrix
	$X$ conformally with the decomposition of $J_p(0)\oplus \Gamma_q$.  In other words, we write
	\begin{equation}\label{eq-X}
		X = \minimatrix{A}{B}{C}{D},
	\end{equation}
	where $A$ is $p \times p$, $B$ is $p \times q$, $C$ is $q\times p$, and $D$ is $q\times q$.  We should pause here to
	remark that the submatrix $A$ just introduced is not the same as the fixed matrix $A$ which appears in the general equation \eqref{eq-DTD},
	which plays no part in what follows.
	
	We next substitute the partitioned matrix \eqref{eq-X} into \eqref{eq-Main} to obtain the system
	\begin{align}
		 AJ_p(0)+J_p(0)A^T &=0, \label{eq-A} \\
		D\Gamma_q+ \Gamma_q D^T &= 0,  \label{eq-D}\\
		B\Gamma_q + J_p(0)C^T &=0, \label{eq-BC1}\\
		CJ_p(0)+\Gamma_q B^T&=0, \label{eq-BC2}
	\end{align}
	of matrix equations which, taken together, are equivalent to \eqref{eq-Main}.  Fortunately,
	equations \eqref{eq-A} and \eqref{eq-D} are easily
	dispatched, since the explicit solutions in these cases have already been described by De Ter\'an and Dopico \cite[p.~52, p.~54]{DTD}.
	
	It therefore remains to characterize the matrices $B$ and $C$ which satisfy \eqref{eq-BC1} and \eqref{eq-BC2}.  This is by far the most
	difficult and involved portion of our work, although our task is made somewhat easier because 
	$B=C=0$ when $p$ is even \cite[p.~60]{DTD}.
	
	In the following, we assume that $p$ is odd.
	Solving for $B$ in \eqref{eq-BC2} yields
	\begin{equation}\label{eq-solveB}
		B=-J_p(0)^TC^T\Gamma_q^{-T}
	\end{equation}	
	and substituting this into \eqref{eq-BC1} gives
	 \begin{equation}\label{eq-C}
		 J_p(0)C^T= J_p(0)^TC^T\Gamma_q^{-T}\Gamma_q.
	 \end{equation}
	 In fact, it is easy to see that the system composed of \eqref{eq-BC1} and \eqref{eq-BC2}, and the system
	 composed of \eqref{eq-solveB} and \eqref{eq-C} are equivalent.  We have already established one direction.  
	 It therefore 
	 suffices to show that if the pair $(B,C)$ solves \eqref{eq-solveB} and \eqref{eq-C},
	 then $(B,C)$ must also solve \eqref{eq-BC1} and \eqref{eq-BC2}.  Solving \eqref{eq-C} for
	 $J_p(0)^T C^T \Gamma_q^{-T}$ and substituting the result into \eqref{eq-solveB}
	 yields \eqref{eq-BC1}.  Taking the transpose of \eqref{eq-solveB} immediately provides \eqref{eq-BC2}.
	 
	 Before proceeding any further, we require a few preliminary lemmas.
	The following fact is well-known amongst those familiar with the
	congruence canonical form (see \cite[p.~1016]{HS1} or \cite[p.~215]{HSCM}).
		
	\begin{Lemma}\label{LemmaCosquare}
		The cosquare $\Gamma_q^{-T} \Gamma_q$ of $\Gamma_q$ is given by
		\begin{equation}\label{eq-GGC}
			\Gamma_q^{-T} \Gamma_q = (-1)^{q+1}\Lambda_q,
		\end{equation}
		where
		\begin{equation}\label{eq-Cosq}
			\Lambda_q :=
			\begin{bmatrix}
				1 & 2 & 2 & \cdots & 2 \\
				0 & 1 & 2 & \cdots  & 2\\
				0 & 0 & 1 & \cdots & 2 \\[-3pt]
				\vdots & \vdots & \vdots & \ddots & \vdots \\
				0 & 0 &0 & \cdots & 1
			\end{bmatrix}.
		\end{equation}
	\end{Lemma}
	
	We also require the explicit reduction of $\Lambda_q$ to its Jordan canonical form since
	the matrix which implements this similarity will be of great interest to us.  Since we could not find
	the following result in the literature, we feel obliged to provide a detailed proof.  In what follows,
	we let $K_q=I_q+J_q(0)$ and adopt the convention that $\binom{n}{k}=0$
	whenever $k < 0$.
	
	\begin{Lemma}\label{LemmaJordan}
		The matrix $\Lambda_q$ satisfies $\Lambda_q = P_q K_q P_q^{-1},$ where $P_q$ is the
		$q \times q$ upper triangular matrix whose $(i,j)$ entry is given by
		\begin{equation*}
			[P_q]_{i,j} = 2^{q-j}(-1)^{j-i} \binom{j-1}{j-i}.
		\end{equation*}
	\end{Lemma}

        \begin{proof}
		Observe that if $j = 1$, then 
		\begin{equation*}
			[\Lambda_{q}P_{q}]_{i,1} = \delta_{i,1} 2^{q-1}= [ P_{q}K_{q}]_{i,1},
		\end{equation*}
		where $\delta_{i,j}$ denotes the Kronecker $\delta$-function,
		since the matrices involved are all upper-triangular.
		For $j\geq 2$, we make note of Identity 168 in \cite{BenJ}, which states that
		\begin{equation}\label{eq:03}
			\sum_{k=0}^{t}(-1)^{k}\binom{n}{k}=(-1)^{t}\binom{n-1}{t},
		\end{equation}
		whenever $t\geq 0$ and $n\geq 1$. This gives us
		\begin{align}
			[\Lambda_{q}P_{q}]_{i,j}&=\sum_{k=1}^{q}[\Lambda_{q}]_{i,k}[P_{q}]_{k,j} \nonumber \\
			&=\sum_{k=i}^{q}[\Lambda_{q}]_{i,k}[P_{q}]_{k,j} \nonumber\\
			&=[\Lambda_{q}]_{i,i}[P_{q}]_{i,j}+\sum_{k=i+1}^{q}[\Lambda_{q}]_{i,k}[P_{q}]_{k,j} \nonumber\\
			&=[P_{q}]_{i,j}+\sum_{k=i+1}^{q}2\cdot2^{q-j}(-1)^{j-k}\binom{j-1}{j-k} \nonumber\\
			&=[P_{q}]_{i,j}+2^{q-j+1}\sum_{k=i+1}^{j}(-1)^{j-k}\binom{j-1}{j-k} \nonumber\\
			&=[P_{q}]_{i,j}+2^{q-j+1}\sum_{k=0}^{j-i-1}(-1)^{k}\binom{j-1}{k} \label{binomial1}\\
			&=[P_{q}]_{i,j}+2^{q-j+1}(-1)^{j-i-1}\binom{j-2}{j-i-1} \label{binomial2}\\
			&=[P_{q}]_{i,j}[K_{q}]_{j,j}+[P_{q}]_{i,j-1}[K_{q}]_{j-1,j} \nonumber \\
			&=\sum_{k=1}^{q}[P_{q}]_{i,k}[K_{q}]_{k,j} \nonumber\\
			&= [P_{q}K_{q}]_{i,j}, \nonumber
		\end{align}
		where the passage from \eqref{binomial1} to \eqref{binomial2} follows by \eqref{eq:03}. Thus 						         
		$[\Lambda_{q}P_{q}]_{i,j}=[P_{q}K_{q}]_{i,j}$ for each pair $(i,j)$, whence $\Lambda_{q}P_{q}=P_{q}K_{q}$,
		as claimed.
                \end{proof}
                
	We next require an explicit description of the inverse of $P_q$.  For the sake of illustration,
	for $q=5$ we obtain the corresponding matrices
	\begin{equation*}\small
            P_5 =
            \left[
                \begin{array}{ccccc}
                 16 & -8 & 4 & -2 & 1 \\
                 0 & 8 & -8 & 6 & -4 \\
                 0 & 0 & 4 & -6 & 6 \\
                 0 & 0 & 0 & 2 & -4 \\
                 0 & 0 & 0 & 0 & 1
                \end{array}
                \right]
            \qquad
            P_5^{-1} =\frac{1}{16}
                            \left[
                \begin{array}{ccccc}
                 1 & 1 & 1 & 1 & 1 \\
                 0 & 2 & 4 & 6 & 8 \\
                 0 & 0 & 4 & 12 & 24 \\
                 0 & 0 & 0 & 8 & 32 \\
                 0 & 0 & 0 & 0 & 16
                \end{array}
                \right].
	\end{equation*}

        \begin{Lemma}\label{LemmaInvese}
            The inverse $P_q^{-1}$ of the matrix $P_q$ is the $q \times q$
            upper triangular matrix whose $(i,j)$ entry is given by
                \begin{equation}\label{eq-PQI}
	                [P_q^{-1}]_{i,j}= 2^{i-q} \binom{j-1}{j-i}.
                \end{equation}
        \end{Lemma}

	\begin{proof}
		Letting $P_q'$ denote the $q \times q$ upper triangular matrix whose entries are given by \eqref{eq-PQI},
		we find that $[P_qP_q']_{i,j} = 0$ for $j < i$ since the matrices involved are both upper-triangular.  For $j \geq i$,
		we appeal to Identity 169 in \cite{BenJ}, which we rewrite as	
		\begin{equation*}
			\sum_{k=0}^{n}(-1)^{k-m}\binom{k}{m}\binom{n}{k}=\delta_{m,n}
		\end{equation*}       		
		for $n\geq m$, to conclude that		
		\begin{align*}
			[P_qP_q']_{i,j}&=\sum_{k=1}^q[P_q]_{i,k}[P_q']_{k,j} \nonumber\\
			&=\sum_{k=1}^{q}2^{q-k}(-1)^{k-i}\binom{k-1}{k-i} \cdot 2^{k-q}\binom{j-1}{j-k}\nonumber\\
			&=\sum_{k=1}^{q}(-1)^{k-i}\binom{k-1}{k-i}\binom{j-1}{j-k}	\nonumber\\
			&=\sum_{k=1}^{j}(-1)^{k-i}\binom{k-1}{k-i}\binom{j-1}{j-k}\nonumber\\
			&=\sum_{k=0}^{j-1}(-1)^{k-i+1}\binom{k}{k-i+1}\binom{j-1}{j-k-1}\nonumber\\
			&=\sum_{k=0}^{j-1}(-1)^{k-(i-1)}\binom{k}{i-1}\binom{j-1}{k}\\
			&= \delta_{i-1,j-1} \\
			&= \delta_{i,j}.
		\end{align*}
		Thus $P_q$ is invertible and $P_q^{-1} = P_q'$.
	\end{proof}

	Returning to \eqref{eq-C} and applying Lemma \ref{LemmaJordan} we obtain
	\begin{equation*}
		J_p(0)C^T P_q=\left(-1\right)^{q-1}J_p(0)^TC^T  P_qK_q.
	\end{equation*}
	In other words, the $p \times q$ matrix 
	\begin{equation}\label{eq-YCP}
		Y = C^T P_q
	\end{equation}
	satisfies the equation
	\begin{equation}\label{eq-yc}
		J_p(0) Y =\left(-1\right)^{q-1}J_p(0)^T Y \big( I_q + J_q(0)\big).
	\end{equation}
	Therefore the entries of $Y$ satisfy
	\begin{equation}\label{eq-YYZ}
		{\footnotesize
	        	\begin{bmatrix}
	        		y_{2,1} & y_{2,2} & \cdots & y_{2,q} \\
	        		y_{3,1} & y_{3,2} & \cdots & y_{3,q} \\
	        		\vdots & \vdots & \ddots & \vdots \\
	        		y_{p,1} & y_{p,2} & \cdots & y_{p,q} \\
	        		0&0&\cdots &0
	        	\end{bmatrix}
	        	= (-1)^{q-1}
		\begin{bmatrix}
	        		0&0&\cdots &0 \\
	        		y_{1,1} & y_{1,1} + y_{1,2} & \cdots & y_{1,q-1} + y_{1,q} \\
	        		y_{2,1} & y_{2,1} + y_{2,2} & \cdots & y_{2,q-1} + y_{2,q} \\
	        		\vdots & \vdots & \ddots & \vdots \\
	        		y_{p-1,1} & y_{p-1,1} + y_{p-1,2} & \cdots & y_{p-1,q-1} + y_{p-1,q} \\
	        	\end{bmatrix},}
		\normalsize
	\end{equation}
	from which we see that
	\begin{equation}\label{yij}
		y_{i,j}=(-1)^{q-1}\big(y_{i-2,j-1}+y_{i-2,j}\big)
	\end{equation}
	holds for $3 \leq i \leq p$ and $2 \leq j \leq q$.  This
	prompts the following lemma.

	\begin{Lemma}
		If $i$ is even, then $y_{i,j}=0$ for $1\leq j\leq q$.  If $i$ is odd, then
		\begin{equation}\label{ybinom}
			y_{i,j}
			= (-1)^{\frac{i-1}{2}(q-1)}\sum_{k=1}^j\binom{\frac{i-1}{2}}{j-k}y_{1,k}.
		\end{equation}
	\end{Lemma}

	\begin{proof}
		Comparing entries in \eqref{eq-YYZ}, we find that $y_{2,j}=0$ for $1\leq j\leq q$.
		In light of \eqref{yij}, it follows that
		\begin{equation*}
			y_{4,j}=(-1)^{q-1} \big(y_{2,j-1}+y_{2,j} \big)=0,
		\end{equation*}
		for $1\leq j\leq q$. Proceeding inductively, we see that $y_{i,j}=0$ whenever $i$ is even.
		
		Now suppose that $i$ is odd.
		For the basis of our induction, observe that \eqref{ybinom} 
		holds trivially when $i+j=2$ (i.e., when $i = j = 1$).  Now suppose that $n \geq 3$ and that
		\eqref{ybinom} holds if $i+j=n-1$.  Under this hypothesis, we wish to show that \eqref{ybinom} also holds if $i+j=n$. 
		In light of \eqref{yij}, we have
		\begin{align}
			y_{i,j}&=(-1)^{q-1} \big(y_{i-2,j-1}+y_{i-2,j} \big)\nonumber\\
			&=(-1)^{q-1}\left[(-1)^{\frac{i-3}{2}(q-1)}\sum_{k=1}^{j-1}{\frac{i-3}{2}\choose j-k-1}y_{1,k}+(-1)^{\frac{i-3}{2}(q-1)}\sum_{k=1}^j{\frac{i-3}{2}\choose j-k}y_{1,k} \right]\nonumber\\
			&=(-1)^{\frac{i-1}{2}(q-1)}\left[\sum_{k=1}^{j-1}{\frac{i-3}{2}\choose j-k-1}y_{1,k}+\sum_{k=1}^j{\frac{i-3}{2}\choose j-k}y_{1,k} \right]\nonumber\\
			&=(-1)^{\frac{i-1}{2}(q-1)}\left[\sum_{k=1}^{j}{\frac{i-3}{2}\choose j-k-1} + {\frac{i-3}{2}\choose j-k} \right]y_{1,k}\nonumber\\
			&=(-1)^{\frac{i-1}{2}(q-1)}\sum_{k=1}^j{\frac{i-1}{2}\choose j-k}y_{1,k}, \nonumber
		\end{align}
		which completes the proof.
	\end{proof}

	Armed with an explicit, entry-by-entry description of the matrix $Y = C^T P_q$, we can construct 
	a convenient basis for the solution space of \eqref{eq-yc}.
	Noting that $Y$ depends upon the $q$ free parameters $y_{1,1}, y_{1,2},\ldots, y_{1,q}$, we introduce a family 
	$W_1,W_2,\ldots,W_q$ of matrices by letting $W_{\ell}$ denote the matrix $Y$ which corresponds to the initial values
	$y_{1,j} = \delta_{j,\ell}$ along the first row and whose remaining entries are determined by \eqref{ybinom}.
	To be more specific, it follows from \eqref{ybinom} that
	\begin{equation}\label{eq:wf}
		[W_{\ell}]_{i,j}=
		\begin{cases}
			0 & \text{if $i$ is even}, \\[5pt]
			\displaystyle (-1)^{\frac{i-1}{2}(q-1)} \binom{\frac{i-1}{2}}{j-\ell} & \text{if $i$ is odd}.
		\end{cases}
	\end{equation}
	In general, for $i$ odd we have
	\begin{align*}
                [W_{\ell}]_{i,j} &=(-1)^{\frac{i-1}{2}(q-1)}\binom{\frac{i-1}{2}}{j-\ell} 
                = (-1)^{\frac{i-1}{2}(q-1)}\binom{\frac{i-1}{2}}{(j-\ell+1)-1}\\
                &=[W_1]_{i,j-(\ell-1)}
                =\left[W_1\big(J_q(0)\big)^{\ell-1}\right]_{i,j},
	\end{align*}
	so that
	\begin{equation}\label{Wm}
                W_{\ell}=W_1\big(J_q(0)\big)^{\ell-1}.
	\end{equation}	
	For example, if $p=q=5$ we obtain
	\begin{equation*}
		Y=\small
		\begin{bmatrix}
	                    y_{1,1}&y_{1,2}&y_{1,3}&y_{1,4}&y_{1,5}\\
	                    0&0&0&0&0\\
	                    y_{1,1}&y_{1,1}+y_{1,2}&y_{1,2}+y_{1,3}&y_{1,3}+y_{1,4}&y_{1,4}+y_{1,5}\\
	                    0&0&0&0&0\\
	                    y_{1,1}&2y_{1,1}+y_{1,2}&y_{1,1}+2y_{1,2}+y_{1,3}&y_{1,2}+2y_{1,3}+y_{1,4}&y_{1,3}+2y_{1,4}+y_{1,5}\\
		\end{bmatrix},
	\end{equation*}
	and
	\begin{equation*}\small
                 W_1=
                    \begin{bmatrix}
                    1&0&0&0&0\\
                    0&0&0&0&0\\
                    1&1&0&0&0\\
                    0&0&0&0&0\\
                    1&2&1&0&0\\
                    \end{bmatrix},  \quad
                 W_2=
                    \begin{bmatrix}
                    0&1&0&0&0\\
                    0&0&0&0&0\\
                    0&1&1&0&0\\
                    0&0&0&0&0\\
                    0&1&2&1&0\\
                    \end{bmatrix},  \quad
                  W_3=
                    \begin{bmatrix}
                    0&0&1&0&0\\
                    0&0&0&0&0\\
                    0&0&1&1&0\\
                    0&0&0&0&0\\
                    0&0&1&2&1\\
                    \end{bmatrix},
                \end{equation*}
                \begin{equation*}\small
                 W_4=
                    \begin{bmatrix}
                    0&0&0&1&0\\
                    0&0&0&0&0\\
                    0&0&0&1&1\\
                    0&0&0&0&0\\
                    0&0&0&1&2\\
                    \end{bmatrix}, \quad
                W_5=
                    \begin{bmatrix}
                    0&0&0&0&1\\
                    0&0&0&0&0\\
                    0&0&0&0&1\\
                    0&0&0&0&0\\
                    0&0&0&0&1\\
                    \end{bmatrix}.
	\end{equation*}

	By \eqref{eq:wf} we can see that $[W_i]_{1,j}=\delta_{i,j}$ for $1\leq i,j\leq q$ whence
	the matrices $W_1,W_2,\ldots,W_q$ are linearly independent.  Considering \eqref{Wm}, 
	we now set $W=W_1$ so that for any solution $Y$ of \eqref{eq-yc} we have
	\begin{equation}\label{Y}
		Y=\sum_{\ell=1}^{q}\alpha_{\ell}W_{\ell}=W\left(\sum_{\ell=1}^{q}\alpha_{\ell}\big(J_q(0)\big)^{\ell-1}\right)
	\end{equation}
	for some constants $\alpha_{1},\alpha_{2},\hdots,\alpha_{q}$.  Moreover, every solution to \eqref{eq-yc} arises in this manner.
	
	Returning to \eqref{eq-YCP} and solving for $C$ yields
	\begin{equation}\label{C}
		\boxed{ C =P_{q}^{-T}\left(\sum_{\ell=1}^{q}\alpha_{\ell}\big(J_q^T(0)\big)^{\ell-1}\right)W^{T},}
	\end{equation}
	where $P_q^{-1}$ is given explicitly by \eqref{eq-PQI} and $W = W_1$ is given by \eqref{eq:wf}.
	Turning our attention now toward \eqref{eq-solveB}, we find that
	\begin{equation}\label{eq-BJCT}
		\boxed{B=-J_p(0)^TC^T\Gamma_q^{-T},}
	\end{equation}
	so that $B$ is completely determined by $C$. In light of \eqref{Y}, the dimension 
	of the solution space of \eqref{eq-BC1} and \eqref{eq-BC2} is $q$, which is consistent with \cite[Lem.~10]{DTD}.
	Furthermore, \cite[Lem.~4]{DTD} and \cite[Lem.~5]{DTD} tell us that the solution sets 
	of  \eqref{eq-A} and \eqref{eq-D} have dimensions $\lceil \frac{p}{2}\rceil$ and $\lfloor \frac{q}{2}\rfloor$, respectively.
	Putting everything together, we conclude that the dimension of the solution space of our original equation 
	\eqref{eq-Main} for odd $p$ is
	\begin{equation*}
                    q+\left\lceil \frac{p}{2}\right\rceil+ \left\lfloor \frac{q}{2}\right\rfloor. 
	\end{equation*}
 
\section{Combinatorial interpretation}
	Although we are now in possession of an explicit description of the solution to \eqref{eq-Main},
	it is worth examining the submatrices $B$ and $C$
	a little bit closer since the entries of these matrices
	display some remarkable combinatorial properties.
	
	Using our closed form solution \eqref{C}, let us continue with our previous series of examples
	corresponding to the parameters $p=q=5$.  In this setting, we have
	\begin{equation*}
	        	C=\tiny
	        	\begin{bmatrix}
	        	 \frac{\alpha_1}{16} & 0 & \frac{\alpha_1}{16} & 0 & \frac{\alpha_1}{16} \\[3pt]
	        	 \frac{1}{16} \left(\alpha_1+2 \alpha_2\right) & 0 & \frac{1}{16} \left(3 \alpha_1+2 \alpha_2\right) & 0 & \frac{1}{16} \left(5 \alpha_1+2 \alpha_2\right) \\[3pt]
	        	 \frac{1}{16} \left(\alpha_1+4 \left(\alpha_2+\alpha_3\right)\right) & 0 & \frac{1}{16} \left(5 \alpha_1+8 \alpha_2+4 \alpha_3\right) & 0 & \frac{1}{16} \left(13 \alpha_1+4 \left(3 \alpha_2+\alpha_3\right)\right) \\[3pt]
	        	 \frac{1}{16} \left(\alpha_1+6 \alpha_2+12 \alpha_3+8 \alpha_4\right) & 0 & \frac{1}{16} \left(7 \alpha_1+18 \alpha_2+20 \alpha_3+8 \alpha_4\right) & 0 & \frac{1}{16} \left(25 \alpha_1+38 \alpha_2+28 \alpha_3+8 \alpha_4\right) \\[3pt]
	        	 \frac{\alpha_1}{16}+\frac{\alpha_2}{2}+\frac{3 \alpha_3}{2}+2 \alpha_4+\alpha_5 & 0 & \frac{9 \alpha_1}{16}+2 \alpha_2+\frac{7 \alpha_3}{2}+3 \alpha_4+\alpha_5 & 0 & \frac{41 \alpha_1}{16}+\frac{11 \alpha_2}{2}+\frac{13 \alpha_3}{2}+4 \alpha_4+\alpha_5
	        	\end{bmatrix}.
	\end{equation*}
	In particular, the simple parametrization at the heart of \eqref{C} is obfuscated 
	by pre- and post- multiplication
	with the somewhat complicated matrices $P_q^{-T}$ and $W^T$, respectively.  
	This results in the rather overwhelming complexity apparent in the preceding example.
	However, upon performing an appropriate reparametrization, we can obtain the significantly 
	simpler matrix
	\begin{equation}\label{Cex}       	
        	C=\small
	        	\begin{bmatrix}
		        	 \beta_1 & 0 & \beta_1 & 0 & \beta_1 \\
		        	 \beta_2 & 0 & \beta_2+2\beta_1 & 0 & \beta_2+4 \beta_1 \\
		        	 \beta_3 & 0 & \beta_3 +2\beta_2+2\beta_1& 0 & \beta_3+4 \beta_2+8 \beta_1 \\
		        	 \beta_4 & 0 & \beta_4 +2\beta_3 +2\beta_2+2\beta_1& 0 & \beta_4+4 \beta_3+8 \beta_2+12 \beta_1 \\
		        	 \beta_5 & 0 & \beta_5+2\beta_4 +2\beta_3 +2\beta_2+2\beta_1 & 0 & \beta_5+4 \beta_4+8 \beta_3+12 \beta_2+16 \beta_1
	        	\end{bmatrix}.
	\end{equation}	 
	We sketch here an independent combinatorial argument
	for determining the exact coefficients which arise in this manner.  Since our motivation for doing so
	is purely aesthetic, we leave many of the tedious details to the reader.

	We begin anew, expanding \eqref{eq-C} and using Lemma \ref{LemmaCosquare} to obtain
	\begin{equation}\footnotesize\label{eq-Recursion}
		\begin{bmatrix}
			c_{1,2} & c_{2,2} &\cdots & c_{q,2} \\
			c_{1,3} & c_{2,3} &\cdots & c_{q,3} \\
			\vdots & \vdots & \ddots & \vdots \\
			c_{1,p} & c_{2,p} & \cdots & c_{q,p} \\
			0 & 0 & \cdots & 0
		\end{bmatrix}
		=
		(-1)^{q-1}\begin{bmatrix}
	 		0 & 0  & \cdots & 0 \\
	 		c_{1,1} & c_{2,1} + 2c_{1,1}& \cdots & c_{q,1}+\sum_{i=1}^{q-1}2c_{i,1} \\
	 		c_{1,2} & c_{2,2}+2c_{1,2} & \cdots & c_{q,2}+\sum_{i=1}^{q-1}2c_{i,2} \\
	 		\vdots & \vdots & \ddots & \vdots \\
	 		c_{1,p-1} & c_{2,p-1}+2c_{1,p-1} & \cdots & c_{q,p-1}+\sum_{i=1}^{q-1}2c_{i,p-1}
		\end{bmatrix}.
	\end{equation}
	Since there are no restrictions placed upon the first column of $C$, we  set 
	\begin{equation*}
		c_{1,1} = \beta_1, \quad
		c_{2,1} = \beta_2,
		\ldots, \quad c_{q,1} = \beta_q,
	\end{equation*}
	where $\beta_1,\beta_2,\ldots,\beta_q$ are free parameters (for our later convenience, we define $c_{i,1}= \beta_i = 0$ for $i > q$).  
	Moreover, it is also clear that 
	\begin{equation}\label{eq-Zeros}
		c_{1,2} = c_{2,2} = \cdots = c_{q,2}=0.
	\end{equation}
	Along similar lines, we set $c_{i,2}=0$ for $i > q$.
	The remaining $c_{i,j}$ can be obtained from the recursion
	\begin{equation}\label{eq-CR}
		c_{i,j} = (-1)^{q-1}\big(c_{i,j-2}+\sum_{\ell=1}^{i-1}2c_{\ell,j-2}\big),
	\end{equation}
	suggested by \eqref{eq-Recursion}.
	Among other things, the initial conditions \eqref{eq-Zeros} and the recursion
	\eqref{eq-CR} tell us that $c_{i,j} = 0$
	whenever $j$ is even.  Noting that 
	\begin{equation*}
		\frac{1+x}{1-x} = 1 + 2x + 2x^2 + 2x^3 + \cdots,
	\end{equation*}
	we define the multivariate generating function 
	$f(x,y) = \sum_{i,j=1}^{\infty} c_{i,j} x^i y^j$
	and use \eqref{eq-CR} to obtain the identity
	\begin{equation*}
		f(x,y) = y h(x) + (-1)^{q-1} y^2 \left( \frac{1+x}{1-x}\right) f(x,y),
	\end{equation*}
	where
	\begin{equation*}
		h(x) = \sum_{n=1}^q \beta_n x^n.
	\end{equation*}
	In other words, $c_{i,j}$ equals the coefficient of $x^i y^j$ in the series expansion of
	\begin{equation}\label{eq-fxy}
		f(x,y) = \frac{yh(x)}{1 - (-1)^{q-1} \left( \frac{1+x}{1-x} \right)y^2}.
	\end{equation}
	
	At this point we take a moment to recall that
	\begin{equation}\label{eq-GF1}
		\left(\frac{1+x}{1-x}\right)^k=\sum_{n=0}^\infty S_k(n)x^n,
	\end{equation}
	where $S_k(n)$ denotes the coordination sequence 
	for a $k$-dimensional cubic lattice (see Table \ref{TableSequence}). 
	To be more specific, $S_k(n)$ represents the 
	number of lattice points in $\mathbb{Z}^k$ at distance $n$ from the origin. 
	For instance, the first row of Table \ref{TableSequence} is $1,0,0,\ldots$ since the trivial lattice 
	contains only the origin.  The second row of Table \ref{TableSequence} is $1,2,2,\ldots$ 
	since the lattice $\Z^1 = \Z$
	contains exactly one point at distance one from $0$ (namely $0$ itself), 
	and exactly two points at distance $n$ from
	$0$ (namely $\pm n$) for $n \geq 1$.  See \cite{Conway} for more information about the numbers $S_k(n)$.
	\begin{table}
	\begin{equation*}
		\small
		\begin{array}{|c||c|c|c|c|c|c|c|c|c|c|}
			\hline
			n&0&1&2&3&4&5&6&7&8&9\\
			\hline
			\hline
			S_0(n)&1&0&0&0&0&0&0&0&0&0\\
			\hline
			S_1(n)&1&2&2&2&2&2&2&2&2&2\\
			\hline
			S_2(n)&1&4&8&12&16&20&24&28&32&36\\
			\hline
			S_3(n)&1&6&18&38&66&102&146&198&258&326\\
			\hline
			S_4(n)&1&8&32&88&192&360&608&952&1408&1992\\
			\hline
			S_5(n)&1&10&50&170&450&1002&1970&3530&5890&9290\\
			\hline
			S_6(n)&1&12&72&292&912&2364&5336&10836&20256&35436\\
			\hline
			S_7(n)&1&14&98&462&1666&4942&12642&28814&59906&115598\\
			\hline
			S_8(n)&1&16&128&688&2816&9424&27008&68464&157184&332688\\
			\hline
			S_{9}(n)&1&18&162&978&4482&16722&53154&148626&374274&864146\\
			\hline
			S_{10}(n)&1&20&200&1340&6800&28004&97880&299660&822560&2060980\\
			\hline
		\end{array}
	\end{equation*}
	\caption{The initial terms of the sequences $S_0(n), S_1(n),\ldots, S_{10}(n)$.    }
	\label{TableSequence}
	\end{table}
	
	Returning to \ref{eq-fxy}, we have
	\begin{align*}
	f(x,y)
	&= yh(x) \sum_{k=0}^{\infty} (-1)^{k(q-1)} \left( \frac{1+x}{1-x} \right)^k y^{2k} \\
	&= \sum_{k=0}^{\infty} (-1)^{k(q-1)} \left(\sum_{r=1}^q \beta_r x^r \sum_{n=0}^\infty S_k(n)x^n \right)y^{2k+1} ,
	\end{align*}
	from which we deduce that the entries of $C$ satisfy
	\begin{equation}\label{eq-cijexplicit}
		c_{i,j}=
		\begin{cases}
		0 & \text{if $j$ is even},\\[5pt] \displaystyle
		(-1)^{\floor{ \frac{j}{2}}(q-1)}\sum_{\ell=0}^{i-1} S_{\lfloor{\frac{j}{2}}\rfloor}(\ell)\beta_{i-\ell}
		& \text{if $j$ is odd}.
		\end{cases}
	\end{equation}

	Since $B=-J_p(0)^TC^T\Gamma_q^{-T}$, we can readily treat $B$ as well.  
	Evaluating this matrix equation entry-by-entry tells us that
	\begin{equation*}
		b_{1,1} = b_{1,2} = \cdots = b_{1,q} = 0
	\end{equation*}
	and
	\begin{equation}\label{eq-bg}
		b_{i,j} = (-1)^q   \Big( c_{j,i-1} + 2\sum_{\ell=1}^{q-j} (-1)^{\ell} c_{j+\ell,i-1} \Big)
	\end{equation}
	for $i\geq 2$.  From \eqref{eq-cijexplicit} and \eqref{eq-bg} we see that $b_{i,j}=0$ whenever $i$ is odd.
	In fact, one can use the recursion \eqref{eq-bg} and the formula \eqref{eq-fxy} for the generating
	function $f(x,y)$ to show that $b_{i,j}$ equals the coefficient of $x^i y^j$ in the series
	expansion of 
	\begin{equation*}
		g(x,y) 
		= (-1)^{q-1} \frac{xf(-y,x)}{y^{q+1}(1+y)}.
	\end{equation*}
	For the even indexed rows of $B$, we may appeal to 
	the generating function identity
	\begin{equation*}
		\left(\frac{1+x}{1-x}\right)^k\cdot\frac{1}{1-x} =\left(\sum_{n=0}^{\infty}S_k(n)x^n\right)
		    \left(\sum_{n=0}^{\infty}x^n\right)
		=\sum_{n=0}^{\infty}G_k(n)x^n,
	\end{equation*}
	where $G_k(n) = \sum_{\ell=0}^n S_k(\ell)$ denotes the number of lattice points in 
	$\mathbb{Z}^k$ at distance $\leq n$ from the origin (see Table \ref{TableG}), to eventually obtain
	\begin{table}
		\begin{equation*}
			\small
			\begin{array}{|c||c|c|c|c|c|c|c|c|c|c|}
				\hline
				n= &0&1&2&3&4&5&6&7&8&9\\
				\hline
				\hline
				G_0(n)&1&1&1&1&1&1&1&1&1&1\\
				\hline
				G_1(n)&1&3&5&7&9&11&13&15&17&19\\
				\hline
				G_2(n)&1&5&13&25&41&61&85&113&145&181\\
				\hline
				G_3(n)&1&7&25&63&129&231&377&575&833&1159\\
				\hline
				G_4(n)&1&9&41&129&321&681&1289&2241&3649&5641\\
				\hline
				G_5(n)&1&11&61&231&681&1683&3653&7183&13073&22363\\
				\hline
				G_6(n)&1&13&85&377&1289&3653&8989&19825&40081&75517\\
				\hline
				G_7(n)&1&15&113&575&2241&7183&19825&48639&108545&224143\\
				\hline
				G_8(n)&1&17&145&833&3649&13073&40081&108545&265729&598417\\
				\hline
				G_{9}(n)&1&19&181&1159&5641&22363&75517&224143&598417&1462563\\
				\hline
				G_{10}(n)&1&21&221&1561&8361&36365&134245&433905&1256465&3317445\\
				\hline
			\end{array}
		\end{equation*}
		\caption{The initial terms of the sequences $G_0(n), G_1(n),\ldots, G_{10}(n)$. }
		\label{TableG}
	\end{table}
\begin{equation*}
	b_{i,j}=
	\begin{cases}
		0 &\text{if $i$ is odd}, \\[5pt]
		(-1)^{j+(k+1)(q-1)} \displaystyle\sum_{\ell=0}^{q-j}G_{k-1}(\ell)\beta_{q-j-\ell+1} &\text{if $i=2k$}.
	\end{cases}
\end{equation*}
For example when $p=7$ and $q=5$, we have
\begin{equation*}
	B=
	\tiny
	\begin{bmatrix}
	 0 & 0 & 0 & 0 & 0 \\
	 -\beta _1-\beta _2-\beta _3-\beta _4-\beta _5 & \beta _1+\beta _2+\beta _3+\beta _4 & -\beta _1-\beta _2-\beta _3 & \beta _1+\beta _2 & -\beta _1 \\
	 0 & 0 & 0 & 0 & 0 \\
	 -9 \beta _1-7 \beta _2-5 \beta _3-3 \beta _4-\beta _5 & 7 \beta _1+5 \beta _2+3 \beta _3+\beta _4 & -5 \beta _1-3 \beta _2-\beta _3 & 3 \beta _1+\beta _2 & -\beta _1 \\
	 0 & 0 & 0 & 0 & 0 \\
	 -41 \beta _1-25 \beta _2-13 \beta _3-5 \beta _4-\beta _5 & 25 \beta _1+13 \beta _2+5 \beta _3+\beta _4 & -13 \beta _1-5 \beta _2-\beta _3 & 5 \beta _1+\beta _2 & -\beta _1 \\
	 0 & 0 & 0 & 0 & 0 \\
	\end{bmatrix}.
\end{equation*}

\bibliographystyle{amsplain}
\bibliography{OMEXAAX2}

\end{document}